\documentclass[a4paper]{amsart}
\usepackage[utf8]{inputenc}
\usepackage[margin=1in]{geometry}

\usepackage{subcaption}
\usepackage{mathtools}
\usepackage{xcolor}
\usepackage{color, ulem}
\usepackage{tikz}
\usepackage{amsmath, amsthm}
\usepackage{comment}
\usepackage{ stmaryrd }
\usepackage{ esint }
\usepackage{dsfont}

\usepackage{amsfonts}
\usepackage{amssymb}

\usepackage{hyperref}
\usepackage{enumitem}

\newcommand{\R}{\mathbb{R}}

\newcommand{\Sb}{\mathbb{S}}
\newcommand{\B}{\mathcal{B}}

\newcommand{\M}{\mathcal{M}}

\newcommand{\dd}{\mathrm{d}}
\newcommand{\e}{\varepsilon}

\newcommand{\ie}{\textit{i.e.}}
\newcommand{\pa}{\partial}
\newcommand{\mph}{\mathfrak{m}_{\varphi}}

\newtheorem{proposition}{Proposition}
\newtheorem{theorem}{Theorem}
\newtheorem{definition}{Definition}

\newtheorem{lemma}{Lemma}
\newtheorem{remark}{Remark}

\def\tb#1{\textcolor{violet}{TB: #1}} 
\def\lb#1{\textcolor{magenta}{LB: #1}} 
\def\jm#1{\textcolor{cyan}{JM: #1}} 

\title[Polyatomic gas modelling with quasi-resonant collisions]{A kinetic model for polyatomic gas with quasi-resonant collisions leading to bi-temperature relaxation processes}

\author[T. Borsoni]{Thomas Borsoni}
\address{T.B.: CERMICS, École des Ponts et Chaussées - Institut Polytechnique de Paris \& Inria Saclay, France \& Sorbonne Universit\'e, Universit\'e Paris Cit\'e, CNRS, Laboratoire Jacques-Louis Lions, LJLL, F-75005 Paris, France}
\email{thomas.borsoni@enpc.fr}

\author[L. Boudin]{Laurent Boudin}
\address{L.B.: Sorbonne Universit\'e, Universit\'e Paris Cit\'e, CNRS, Laboratoire Jacques-Louis Lions, LJLL, F-75005 Paris, France}
\email{laurent.boudin@sorbonne-universite.fr}

\author[J. Mathiaud]{Julien Mathiaud}
\address{J.M.: Université de Rennes, CNRS, IRMAR, UMR 6625, Rennes, France }
\email{julien.mathiaud@univ-rennes.fr}

\author[F. Salvarani]{Francesco Salvarani}
\address{F.S.: De Vinci Higher Education, De Vinci Research Center, Paris, France \& Dipartimento di Matematica ``F. Casorati'', Universit\`a degli Studi di Pavia, Via Ferrata 1, 27100 Pavia,
Italy}
\email{francesco.salvarani@unipv.it}

\thanks{This article has been partially funded by the Starting Grant ERC project HighLEAP (High-dimensional mathematical methods for LargE Agent and Particle systems), awarded to Virginie Ehrlacher. Partial support by INdAM-GNFM is also acknowledged.}

\date{\today}

\keywords{Boltzmann equation; Polyatomic gases; resonant collisions; quasi-resonant collisions; near-resonance; Landau--Teller equations.}

\begin{document}

\maketitle
 
\begin{abstract}
In this article, we extend the Boltzmann framework for polyatomic gases by introducing quasi-resonant kernels, which relax resonant interactions, for which kinetic and internal energies are separately conserved and lead to equilibrium states with two temperatures. We establish an $H$-theorem and analyze the quasi-resonant model's asymptotic behaviour, demonstrating a two-phase relaxation process: an initial convergence towards a two-temperature Maxwellian state followed by gradual relaxation of the two temperatures towards each other. Numerical simulations validate our theoretical predictions. The notion of quasi-resonance provides the first rigorous framework of a Boltzmann dynamics for which the distribution is at all times close to a multi-temperature Maxwellian, relaxing towards a one-temperature Maxwellian.
\end{abstract}



\section{Introduction}

In this article, we propose a rigorous Boltzmann framework for modelling a polyatomic gas whose molecules undergo quasi-resonant collisions, relying on collision selection through the collision kernel. We discuss the behaviour of the corresponding homogeneous dynamic, for which we expect the distribution to be at all times close to a two-temperature Maxwellian, which temperatures relax towards each other, following a Landau-Teller type equation. This is finally numerically verified.

The study of polyatomic gases has garnered significant interest in recent years, with numerous researchers contributing to this field.
Polyatomic gases differ from monatomic gases because their molecular structure cannot be fully described using only kinematic variables such as position and velocity. To account for these differences, an additional independent variable representing internal structure is required. The first works in this direction date back to the 1950s with Wang Chang and Uhlenbeck~\cite{chang1951transport} and Taxman~\cite{taxman1958classical}.  A significant contribution was provided by Borgnakke and Larsen in~\cite{bor-lar-75}, introducing a numerically well-suited procedure. In this paper, we describe the internal structure of the polyatomic molecules through their internal energy level, denoted by $I$, and belonging to $\R_+$ endowed with a measure $\mu$, called the internal energy law. We assume $\mu$ to have a nonnegative and continuous density $\varphi$ with respect to the Lebesgue measure as in the seminal works \cite{MR1277241, desvillettes1997modele, MR2118066}, \ie
\[\dd \mu(I) = \varphi(I)\,\dd I, \qquad I>0. \]
Some typical choices of $\varphi$ are $\varphi : I \mapsto I^{\delta/2-1}$, corresponding to a polytropic gas with $3 + \delta$ number of degrees of freedom, and $\varphi : I \mapsto \lceil I \rceil$, corresponding to a diatomic molecule, whose rotation is classically described via a rigid-rotor model, and vibration is described via a quantum harmonic oscillator model, for which we refer to~\cite{bisi2025modelling}.
Back in the 1960's, it was already shown in~\cite{sharma1969near} that vibrational energy transfers between $\mathrm{CO}_2$ isotopes would exhibit \textit{near-resonance} behaviour, in the sense that collision pathways for which the post-collision vibrational energy is close to the pre-collisional one are largely dominant.
 This was later supported by~\cite{stephenson1972temperature} for a mixture involving in particular $\mathrm{CO}_2$ molecules. Detailed computations on collision cross-sections involving the $\mathrm{CO}_2$ molecule recently conducted~\cite{lom-fag-pac-gro-15} also showed such a behaviour. Similar results for other molecules, either with experimental data or ab initio calculations, have also been pointed out since then~\cite{brenner1981near, sharma2006near}. Early discussions on resonant vibrational energy transfers are also found in~\cite{mahan1967resonant, nikitin1962resonance}. Ab initio calculations conducted in~\cite{lanza2014near} on the rotational energy transfer in  $\mathrm{HCl-H}_2$ collisions were shown to be mostly near-resonant, and to quote the authors of the latter work, ``Near resonant energy transfer process can be an important, often dominant, rotational energy transfer pathway at low temperatures''.

As such, despite their broad applicability and widespread use, polyatomic Boltzmann models with standard cross-sections, such as Variable Hard Spheres, exhibit inconsistencies with experimental data and theoretical calculations in the above mentioned scenarios.
As noted in \cite{lom-fag-pac-gro-15}, the kinetic modelling of carbon dioxide presents some peculiar challenges due to the emergence of two distinct temperatures associated with translational and rotational energies, which do not mutually interact.
Consequently, in the mathematical models describing this framework, the kinetic and internal energy of colliding particles must be separately, or almost separately,  conserved. 

The first advancements in this direction were made by Aoki \textit{et al.} \cite{MR4010777}, who extended the ellipsoidal statistical model of the Boltzmann equation for a polyatomic gas with constant specific heats (proposed in \cite{andries}) to a polyatomic gas with temperature-dependent specific heats.
Some other articles have then studied the complex dynamics of such gases \cite{PhysRevFluids.3.023401, PhysRevE.102.023104, MR4010777, aoki2021note, Kosuge2023} in an ES-BGK framework, highlighting several properties of resonant collisions. 

Since the 1960s, several authors have proposed kinetic models that preserve fundamental global properties of the original Boltzmann equation while neglecting detailed collision mechanisms \cite{ros-maz, morse1964kinetic, holway1966new, rykov1975model}. 

Indeed, transition probabilities, which are essential in the Boltzmann description, are difficult to compute \textit{ab initio}, and can be derived from phenomenological assumptions \cite{morse1964kinetic} or obtained via molecular dynamics simulations for specific gases \cite{nagnibeda2009non}.

More recently, a Boltzmann model to describe resonant collisions has been introduced in \cite{boudin:hal-03629556}. The equilibrium solution of the model depends on two temperatures, and several properties, including the $H$-theorem, have been deduced. Given the recent introduction of the corresponding model in the literature, the mathematical study of resonant collision kernels remains less developed compared to the standard Boltzmann operator. Nonetheless, significant progress has been made, including investigations into the compactness of the linearized operator \cite{borsonicompact, borsonicompactcorrect}. Let us briefly discuss that resonant model. As already stated, a collision is said to be resonant when the microscopic kinetic and internal energies are \emph{separately} conserved during the collision together with the microscopic momentum.
Hence, formally, the model rises as a singular case of polyatomic collisions exchanging momentum, kinetic energy and internal energy. Since the conservation laws are directly related to the support of the collision kernel, the resonant property of the collisions can be provided by formally putting the internal energy conservation in the collision kernel as a Dirac mass. 
With this viewpoint, a natural way to define quasi-resonant collision kernels is to replace the Dirac mass by a 
nonnegative approximation of that Dirac mass in the kernel expression. 

In this work, we propose a Boltzmann model with quasi-resonant collisions, providing a more realistic framework than the resonant one, and allowing to capture slow bi-temperature relaxations, during which the distribution is at all times close to a bi-temperature Maxwellian.
From a theoretical perspective, incorporating quasi-resonant kernels extends classical thermal relaxation models within the Boltzmann framework and provides insight into transitions between dynamical regimes in kinetic transport. This approach bridges the gap between idealized resonant interactions and the complexity of real molecular systems, offering a more general and physically consistent description of energy exchange. 

The article is structured as follows.

In the next Section, we introduce the model and establish a precise definition of quasi-resonance quantified by
a quasi-resonance parameter $\e > 0$. We then derive an $H$-theorem for this model and investigate the resonant asymptotics as $\e$ vanishes, recovering the version of the resonant model used in \cite{borsonicompactcorrect} 
equipped with the correct 
internal energy measure. 

In Section~\ref{sec:2kasires}, we explore the implications of these findings and discuss the expected properties of the model. We argue that a homogeneous quasi-resonant Boltzmann dynamics likely follows a two-phase evolution. Initially, in the short-term regime, the distribution behaves similarly to the fully resonant case and relaxes toward a two-temperature Maxwellian equilibrium. Subsequently, in the long-term regime, the distribution remains in this form while the kinetic and internal temperatures gradually converge towards each other. Assuming that latter behaviour holds, we explicitly derive an ODE system which approximately governs the evolution of the kinetic and internal temperatures during the second phase. This system, valid for a broad class of collision kernels, is of Landau--Teller-type, see, for instance, \cite{nikitin200870}.

Finally, in Section~\ref{section:numerical_quasires}, we conduct numerical simulations to validate our theoretical predictions. Specifically, we verify that the kinetic and internal temperatures approximately follow the derived Landau--Teller-type ODE system, and conduct a preliminary quantitative study of the behaviour of the approximation as the quasi-resonance parameter $\e$ goes to $0$.

\section{Boltzmann model with quasi-resonant collisions} \label{sec:1kasires}

The approach developed here is based on a polyatomic kinetic framework wherein quasi-resonant collisions are selected
by truncating the cross-section. 

We study a single gas composed of polyatomic molecules, which are described, at the microscopic level, by their mass $m>0$, velocity $v \in \R^3$ and internal energy level $I \in \R_+$. We then consider two colliding molecules with respective incoming velocities $v$, $v_*\in\R^3$ and internal energies $I$, $I_* \in \R_+$ and their outgoing counterparts $v'$, $v_*'$, $I'$, $I_*'$. Since the momentum and total energy of the two-particle (isolated) system are conserved during the collision, the following microscopic conservation laws hold
\begin{align}
    v + v_* &= v' + v'_*, \label{e:momconserv} \\
 \frac12 m   |v|^2 + \frac12 m  |v_*|^2 + I + I_* &= \frac12 m  |v'|^2 + \frac12 m  |v'_*|^2 + I' + I'_*. \label{e:totalenrgconserv}
\end{align}
It is important to notice that the total energy in the center of mass of the molecules
\begin{equation} \label{e:defE}
    E=\frac{m}{4}|v-v_*|^2 + I + I_*=\frac{m}{4}|v'-v_*'|^2 + I' + I_*'
\end{equation} 
is also conserved during the collision. Then \eqref{e:momconserv}--\eqref{e:totalenrgconserv} ensure the existence of $\sigma\in\Sb^2$ such that 
\begin{equation} \label{e:prepostvit}
v' = \frac{v+v_*}{2} + \frac{2}{\sqrt{m}}\sqrt{E-(I'+I'_*)} \, \sigma, \qquad v'_* = \frac{v+v_*}{2} -
\frac{2}{\sqrt{m}} \sqrt{E-(I'+I'_*)} \, \sigma, 
\end{equation}
where $I'+I'_* \leq E$. We moreover introduce the following notations
\begin{equation} \label{e:defRRprime}
    R  = \frac{m|v'-v_*'|^2}{4E} = 1- \frac{I' + I_*'}{E}, 
    \qquad R' = \frac{m|v-v_*|^2}{4E}= 1- \frac{I + I_*}{E},
\end{equation}
as in the Borgnakke-Larsen description, see \cite{MR1277241} for instance. At the microscopic level, resonant collisions, for which the kinetic and internal energy are separately conserved, can thus be characterized by the property 
$$R=R'.$$
In the quasi-resonant case, both the kinetic and internal energies are separately almost conserved, \ie
\[ |v'|^2 + |v'_*|^2 \approx |v|^2 + |v_*|^2,\qquad I' + I'_* \approx I + I_*,\]
or equivalently, 
\[
R \approx R'.
\]
Our approach, detailed in Subsection~\ref{ss:qrtrunc}, consists in considering some reference polyatomic collisions, and removing the ones which are deemed too far from being resonant. 

\subsection{Quasi-resonance by collision kernel truncation} \label{ss:qrtrunc}

In the Boltzmann model, the collisional properties are embedded in the collision kernel. It is a positive function whose support (related to its domain of definition) is provided by the microscopic momentum and energy conservation laws, and whose values formally correspond to some density of probability of possible collisions. Therefore, it is equivalent to state that a collision is possible, and that the collision kernel takes a positive value on the variables involved in the latter. 
We consider a reference collision kernel $B$ whose support characterizes standard polyatomic collisions. From that reference kernel, we define the family of quasi-resonant collision kernels $(B_{\e})_{\e > 0}$, indexed by the quasi-resonant parameter $\e>0$, as the product of $B$ with a cut-off function $\chi_\e$. This function, defined below in \eqref{e:defchiepsi}, allows the removal of collisions deemed too far from being resonant relatively to the parameter $\e$, and is built in such a way that it approximates a Dirac mass (corresponding to the resonant case) when $\e$ vanishes. Figure~\ref{fig:schematic} provides a schematic representation of the situation.

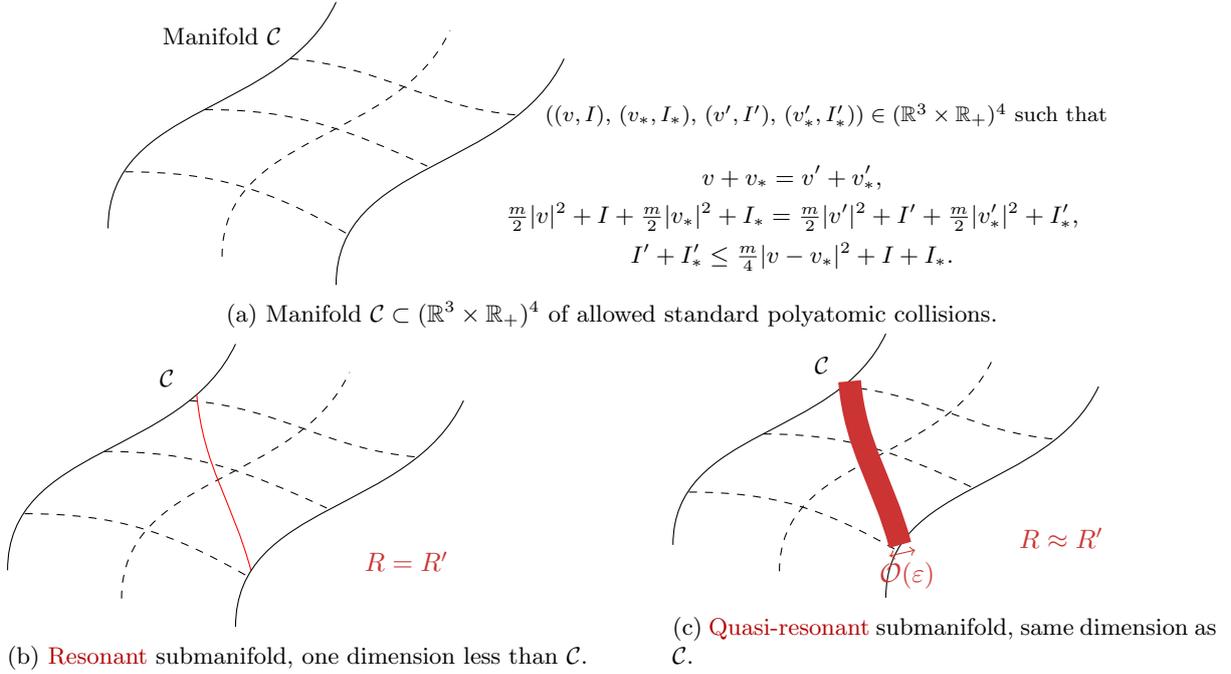
\begin{figure}[!ht]
    \centering
    \begin{subfigure}{0.9\textwidth}
        \centering
\begin{tikzpicture}[scale = .75]
\draw[fill=black!50] (2,3.4) node{\small Manifold $\mathcal{C}$};
    \draw[black] (0,0) to[out=90, in=-115]  (4, 4);
    \draw[black] (4,-1) to[out=90, in=-115]  (8, 3);
    \draw[dashed, black] (2,-.5) to[out=90, in=-115]  (6, 3.5);
    
    \draw[dashed, black] (.3,1) to[out=0, in=150]  (4.2, -.1);
    
    \draw[dashed, black] (1.7,2.1) to[out=0, in=155]  (5.6, 1.1);
    
    \draw[dashed, black] (3.2,3) to[out=-5, in=175]  (7.2, 2);
\draw[fill=black!50] (12.6,2) node{{\footnotesize $((v,I), \, (v_*,I_*), \, (v',I'), \, (v'_*,I'_*) ) \in (\R^3 \times \R_+)^4$  such that}};
\draw[fill=black!50] (12,.9) node{{\small $v+ v_* = v' + v'_*$,}};
\draw[fill=black!50] (12,.2) node{{\small $\frac{m}{2} |v|^2 + I + \frac{m}{2} |v_*|^2 + I_*  = \frac{m}{2} |v'|^2 + I'  + \frac{m}{2} |v'_*|^2 + I'_* $,}};
\draw[fill=black!50] (12,-.5) node{{\small $I' + I'_* \leq \frac{m}{4}|v-v_*|^2 + I + I_*$.}};
\end{tikzpicture}
\caption{Manifold $\mathcal{C} \subset (\R^3 \times \R_+)^4$ of allowed standard polyatomic collisions.}
    \label{fig:poly_mani}
\end{subfigure}
\begin{subfigure}{0.48\textwidth}
\begin{tikzpicture}[scale = .75]
\draw[fill=black!50] (2.8,3.4) node{\small $\mathcal{C}$};
    \draw[black] (0,0) to[out=90, in=-115]  (4, 4);
    \draw[black] (4,-1) to[out=90, in=-115]  (8, 3);
    \draw[dashed, black] (2,-.5) to[out=90, in=-115]  (6, 3.5);
    
    \draw[dashed, black] (.3,1) to[out=0, in=150]  (4.2, -.1);
    
    \draw[dashed, black] (1.7,2.1) to[out=0, in=155]  (5.6, 1.1);
    
    \draw[dashed, black] (3.2,3) to[out=-5, in=175]  (7.2, 2);

\draw[red] (3.32,3.1) to[out=-85, in=105]  (4.27, 0);

\draw[red!60!gray, below] (7, .5) node{$\, R = R' \, $};

\end{tikzpicture}
\caption{\textcolor{red!70!black}{Resonant} submanifold, one dimension less than $\mathcal{C}$.}
    \label{fig:res_mani}
\end{subfigure}
    \hfill
\begin{subfigure}{0.45\textwidth}
\begin{tikzpicture}[scale = .7]
\draw[fill=black!50] (2.8,3.4) node{\small $\mathcal{C}$};
    \draw[black] (0,0) to[out=90, in=-115]  (4, 4);
    \draw[black] (4,-1) to[out=90, in=-115]  (8, 3);
    \draw[dashed, black] (2,-.5) to[out=90, in=-115]  (6, 3.5);
    
    \draw[dashed, black] (.3,1) to[out=0, in=150]  (4.2, -.1);
    
    \draw[dashed, black] (1.7,2.1) to[out=0, in=155]  (5.6, 1.1);
    
    \draw[dashed, black] (3.2,3) to[out=-5, in=175]  (7.2, 2);

\draw[line width = 3mm, red!60!gray] (3.32,3.1) to[out=-85, in=105]   (4.27, 0);

\draw[<->, red!60!gray] (4.07, -.21) -- (4.55, -.1);

\draw[red!60!gray, below] (4.4, -.15) node{$\mathcal{O}(\e)$};

\draw[red!60!gray, below] (7.3, .5) node{$\, R \approx R' \, $};

\end{tikzpicture}
\caption{\textcolor{red!70!black}{Quasi-resonant} submanifold, same dimension as $\mathcal{C}$.}
    \label{fig:QR_mani}
    \end{subfigure}
    \caption{Standard, resonant and quasi-resonant manifolds of allowed collision quadruplets in $(\R^3 \times \R_+)^4$. }
    \label{fig:schematic}
\end{figure}

In the standard polyatomic setting, the collision kernel seen as a measure on $(\R^3 \times \R_+)^4$ is supported on a manifold $\mathcal{C}$, as shown on Subfigure~\ref{fig:poly_mani}. Then the intersection of $\mathcal C$ with the set of quadruplets $\left((v,I),(v_*,I_*),(v',I'),(v_*',I_*') \right)$ such that $R=R'$ corresponds to the resonant manifold seen on Subfigure~\ref{fig:res_mani}. In the quasi-resonant setting drawn on Subfigure~\ref{fig:QR_mani}, the collision kernel support coincides with the support of the cut-off function $\chi_\e$, which is a narrow strip within $\mathcal{C}$ surrounding the resonant submanifold, with a width of order $\e$. The quasi-resonant model is then built such that, as $\e$ goes to $0$, the quasi-resonant collision kernel measure gets closer to the resonant collision kernel measure. 

To make this argument more precise, we introduce the following formalism.

\subsubsection*{Reference collision kernel} We introduce $B:= B(v,v_*,I,I_*,I',I'_*,\sigma)\ge 0$ a reference polyatomic collision kernel,
in the sense that it satisfies, for almost every $v$, $v_* \in \R^3$, $I$, $I_*$, $I'$, $I_*'\in\R_+$, $\sigma\in \Sb^2$, a symmetry property
\begin{equation} \label{e:genBsymm}
B(v,v_*,I,I_*,I',I'_*,\sigma) =  B(v_*,v,I_*,I,I'_*,I',-\sigma),
\end{equation}
a micro-reversibility assumption
\begin{equation} \label{e:genBmicrorev}
|v-v_*| B(v,v_*,I,I_*,I',I'_*,\sigma) = |v'-v'_*|B(v',v'_*,I',I'_*,I,I_*,\sigma'), 
\end{equation}
and a positivity condition to authorize only physical collisions, that is
\begin{equation} \label{e:genB>0}
 I' + I'_* \leq E  ~\iff~ B(v,v_*,I,I_*,I',I'_*,\sigma) > 0,
\end{equation}
remembering that $E$ is defined in \eqref{e:defE}. These assumptions may be found for instance in \cite{borsoni2022general}, and are compatible with other standard formulations as in \cite{desvillettes1997modele,gamba2020cauchy}, as soon as the corresponding Boltzmann operator is defined accordingly.

A possible reference collision kernel, which we later use for explicit computations in Sections~\ref{sec:2kasires}--\ref{section:numerical_quasires}, would be
\begin{multline*} 
B(v,v_*,I,I_*,I',I'_*,\sigma) = C_B \; b(\sigma \cdot \sigma') \; \left[\frac{m}{4}|v-v_*|^2\right]^{\kappa_k-1/2}  \;\left[\frac{m}{4}|v'-v_*'|^2\right]^{\kappa_k} \\ (I+I_*)^{\kappa_i}  \; (I'+I_*')^{\kappa_i} \; E^{\gamma} \; \frac{\mathds{1}_{[I'+I'_* \leq E]} }{\mathfrak{m}_\varphi (E)},
\end{multline*} 
where 
the velocities $v'$ and $v'_*$ are functions of the variables of $B$ given by \eqref{e:prepostvit},  $\kappa_k$, $\kappa_i$ and $\gamma$ are real numbers, $C_B > 0$ is a (dimensional) constant, and $b$ is an angular kernel satisfying, for $\omega\in\Sb^2$,
\begin{equation*} 
\int_{\Sb^2} b(\sigma \cdot \omega) \, \dd \sigma = 1.
\end{equation*}
Eventually, the notation $\mathds{1}_X$ stands for the characteristic function of the set $X$, and $\mathfrak{m}_\varphi (E)$ is the $L^1$-norm of $\mathds{1}_{[I'+I'_* \leq E]}$ on $(0,+\infty)^2$ with respect to the internal energy distribution measure, that is 
\begin{equation} \label{e:gfunction}
\mathfrak{m}_\varphi (E) = \iint_{(0,+\infty)^2} \, \mathds{1}_{[I'+I'_* \leq E]} \, \varphi(I')\, \varphi(I_*')\, \dd I' \, \dd I_*'.
\end{equation}
We emphasize again that the above explicit choice of $B$ is not required at all in this current section. 

\subsubsection*{Truncation function and quasi-resonant collision kernel}  
Let us now build the quasi-resonant family of kernels $(B_\e)_{\e > 0}$ from the reference kernel $B$. In order to discriminate whether a collision is deemed quasi-resonant or not, relatively to $\e$, we introduce the family of cut-off functions $(\chi_\e)_{\e>0}$, defined, for all $\e>0$, by 
\begin{equation} \label{e:defchiepsi}
    \chi_{\e}(R,R') = \frac{c_\eta(R,R')}{2 \e} \, \mathds{1}_{\left[|\eta(R)-\eta(R')| \leq \e \right]}, \qquad \quad R,R' \in (0,1).
\end{equation}
Here, $\eta$ is a $\mathcal{C}^\infty$-diffeomorphism from $(0,1)$ to $\R$, and $c_{\eta}$ is a $\mathcal{C}^1$ normalizing factor, satisfying a symmetry property
\begin{equation} \label{e:symmetry_ceta}
    c_{\eta}(R,R') = c_{\eta}(R',R), \qquad R,R' \in (0,1),
\end{equation}
which is required to preserve the micro-reversibility of $B_\e$, and a diagonal condition
\begin{equation} \label{e:reson-c_eta}
    c_\eta(R,R)  = \eta'(R), \qquad R \in (0,1),
\end{equation}
which is crucial to recover the resonant asymptotics when $\e$ goes to $0$, as we shall see below. 
Possible choices for $c_\eta$ would be $c_\eta(R,R') = \sqrt{\eta'(R)\eta'(R')}$ or $c_\eta(R,R') = (\eta'(R) + \eta'(R'))/2$, but the upcoming computations can be carried out by keeping $c_\eta$ general. Indeed, we shall see that the study of the $\e$-vanishing asymptotics only requires, at the first non-zero order with respect to $\e$, the knowledge of conditions~\eqref{e:symmetry_ceta}--\eqref{e:reson-c_eta}.

We emphasize that the cut-off functions $(\chi_\e)$ fully characterize the notion of quasi-resonance. Indeed, at fixed $\e>0$, the collisions deemed to be quasi-resonant are exactly the ones for which $\chi_\e(R,R')>0$. Moreover, the support of $\chi_\e$ goes to $(0,1)^2$ (all collisions are allowed: standard polyatomic case) when $\e \to \infty$, and that $\chi_\e$ approximates the Dirac at $R=R'$ (resonant collisions) when $\e \to 0$.

\begin{remark} \label{r:deriv-c_eta}
The symmetry property of $c_\eta$ implies that $\pa_1 c_\eta(R,R')=\pa_2 c_\eta(R',R)$ for any $R$, $R'\in (0,1)$. Besides, by differentiating \eqref{e:reson-c_eta}, we get, for any $R\in (0,1)$, $\pa_1 c_\eta(R,R)+\pa_2 c_\eta(R,R)=\eta''(R)$, so that, for any $R\in (0,1)$, 
\begin{equation} \label{e:deriv-c_eta}
    \pa_1 c_\eta(R,R)=\pa_2 c_\eta(R,R) = \frac{\eta''(R)}{2}.
\end{equation}
\end{remark}

In the expression \eqref{e:defchiepsi} of the truncation function, the $\mathcal{C}^\infty$-diffeomorphism $\eta$ is used to enhance the model flexibility and prevent boundary issues. Indeed, using $\mathds{1}_{[|R-R'|\leq \e]}$ in \eqref{e:defchiepsi} would be problematic. For instance, when $R' < \e$ is fixed, the set $\{R \in (0,1) \mid |R-R'|\leq \e \}$ becomes asymmetric due to the constraint $R > 0$. Using instead $\mathds{1}_{[|\eta(R)-\eta(R')|\leq \e]}$ avoids this issue. 

A typical choice of $\eta$, later used in the explicit computations of Sections~\ref{sec:2kasires}--\ref{section:numerical_quasires}, would be
\begin{equation*} 
    \eta: (0,1)\to \R, \quad R\mapsto \log \left(\frac{R}{1-R} \right),
\end{equation*}
which is indeed an (increasing) $\mathcal{C}^\infty$-diffeomorphism from $(0,1)$ to $\R$. The distance from resonance measured through the quantity $\eta(R)-\eta(R')$ corresponds in this case to
$$\eta(R)-\eta(R') = \left(\log E_k' - \log E_i'\right) - \left(\log E_k-\log E_i \right)
 = \left(\log E_k' - \log E_k\right) - \left(\log E_i'-\log E_i \right).$$
That difference hence allows to simultaneously quantify the discrepancy between pre- and post-collisional energies, but also between kinetic and internal ones. Moreover, due to the presence of the $\log$ function, the comparison is made through ratios instead of differences.

\medskip

We are now in a position to define the family of quasi-resonant collision kernels $(B_{\e})_{\e > 0}$ associated to the reference kernel $B$ and the truncation family $(\chi_\e)_{\e >0}$, by, for any $\e > 0$, $v$, $v_* \in \R^3$, $\sigma \in \Sb^2$ and $I$, $I_*$, $I'$, $I'_* \in \R_+$,
\begin{equation} \label{e:defBepsi}
    B_{\e}(v,v_*,I,I_*,I',I'_*,\sigma) = B(v,v_*,I,I_*,I',I'_*,\sigma) \, \chi_{\e}(R,R'),
\end{equation}
where $R$ and $R'$ are given by \eqref{e:defRRprime}. Thanks to \eqref{e:symmetry_ceta}, for any $\e > 0$, the kernel $B_\e$ satisfies the symmetry and micro-reversibility properties, for any $v,v_*$, $\sigma$, $I, I_*,I',I'_*$,
\begin{align}
B_\e(v,v_*,I,I_*,I',I'_*,\sigma) &=  B_\e(v_*,v,I_*,I,I'_*,I',-\sigma),  \label{e:genBepssymm}\\
|v-v_*| B_\e(v,v_*,I,I_*,I',I'_*,\sigma) &= |v'-v'_*|B_\e(v',v'_*,I',I'_*,I,I_*,\sigma'), \label{e:genBepsrev}
\end{align}
and thanks to \eqref{e:genB>0}, $B_\e$ satisfies the following positivity property, which differs from \eqref{e:genB>0} through the involvement of $\eta$ and $\e$,
\begin{equation} \label{e:positivity_Beps_quasires}
    \left\{I'+I'_* \leq E \quad \text{and} \quad |\eta(R) - \eta(R')| \leq \e \right\} ~\iff~ B_\e(v,v_*,I,I_*,I',I'_*,\sigma) > 0.
\end{equation}


\subsection{Quasi-resonant Boltzmann operator and \textit{\textbf{H}}-theorem} \label{ss:Boltzmann_quasires}

We consider a family of quasi-resonant collision kernels $(B_\e)_{\e > 0}$ defined as in the previous subsection, and denote by $(Q_\e)_{\e > 0}$ the associated family of Boltzmann quasi-resonant collision operators. For $\e > 0$, $Q_\e$ is defined for any density $f \equiv f(v,I) \geq 0$ for which it makes sense, and almost every $v \in \R^3$, $I \in \R_+$, by
\begin{multline} \label{e:quasires_boltzmann_op}
Q_\e(f,f)(v,I) \phantom{\int_{\R^3}}\\
=\int_{\R^3} \iiint_{(0,+\infty)^3} \int_{\mathbb{S}^2} \left(f' f'_* - f f_* \right) B_{\e}(v,v_*,I,I_*,I',I'_*,\sigma) \, \dd \sigma  \, \varphi(I') \, \dd I'  \, \varphi(I'_*) \, \dd I'_* \, \varphi(I_*) \, \dd I_* \, \dd v_*,
\end{multline}
where, in the above equation, we use the standard shortcuts $f \equiv f(v,I)$,  $f_* \equiv f(v_*,I_*)$,  $f' \equiv f(v',I')$,  $f_*' \equiv f(v'_*,I'_*)$, and $v'$, $v'_*$ are defined by \eqref{e:prepostvit}.

\subsubsection*{Weak form of the operator}
Since the kernel $B_{\e}$ satisfies the usual polyatomic symmetry and micro-reversibility conditions, we immediately obtain (see for instance \cite{borsoni2022general}) that, for any test-function $\psi \equiv \psi(v,I)$ such that the following makes sense, we have
\begin{multline} \label{eq:weakform}
\int_{\R^3} \int_0^{+\infty} \psi(v,I) \, Q_\e(f,f)(v,I) \, \varphi(I) \, \dd I \, \dd v  \\
= \frac12\iint_{(\R^3)^2} \iiiint_{(0,+\infty)^4} \int_{\mathbb{S}^2} f f_* \, \left(\psi' + \psi'_* - \psi - \psi_* \right) \phantom{\iint_{(\R^3)^2}} \phantom{\iint_{(\R^3)^2}} \\
\phantom{\iint_{(\R^3)^2}} B_{\e}(v,v_*,I,I_*,I',I'_*,\sigma) \, \dd \sigma  \, \varphi(I') \, \dd I'  \, \varphi(I'_*) \, \dd I'_* \, \varphi(I) \, \dd I  \, \varphi(I_*) \, \dd I_* \, \dd v \, \dd v_*,
\end{multline}
where we use the same shortcuts as in \eqref{e:quasires_boltzmann_op}. We deduce from its weak form \eqref{eq:weakform} that, as usual, $Q_\e$ conserves, at least formally, the mass, momentum and total energy. More precisely, for any $f \equiv f(v,I)$ such that it makes sense, the following property holds
\begin{equation} \label{e:conservationsQeps_quasires}
    \int_{\R^3} \int_0^{+\infty} Q_\e (f,f) (v,I) \begin{pmatrix}
        1 \\ v \\ \dfrac12 m |v|^2 + I
    \end{pmatrix} \, \varphi(I) \, \dd I  \,\dd v= 0.
\end{equation}
Notice that the conservation properties of the quasi-resonant operators are the same as those in the non-resonant polyatomic case. On the other hand, in the resonant case, kinetic and internal energy are conserved separately.

\subsubsection*{Collision invariants and $H$ theorem} 
We now focus on the basic properties of the quasi-resonant kernel.
\begin{definition}
A function $\psi : \R^3 \times \R_+ \to \R$ is said to be an $\e$-quasi-resonant collision invariant if $\psi$ is measurable and satisfies, for any $v$, $v_* \in \R^3$, $\sigma \in \Sb^2$ and $I$, $I_*$, $I'$, $I'_* \in \R_+$, 
\begin{equation} \label{eq:collision_inv_quasires}
\left(I'+I'_* \leq E \quad \text{and} \quad |\eta(R) - \eta(R')| \leq \e \right)  \quad \implies \quad \left(\psi(v',I') + \psi(v'_*,I'_*) = \psi(v,I) + \psi(v_*,I_*) \right),
\end{equation}
where $E$ is given by \eqref{e:defE}, $v'$, $v'_*$  by \eqref{e:prepostvit}, and $R$ and $R'$ by \eqref{e:defRRprime}.
\end{definition}

The definition of $\e$-quasi-resonant collision invariants slightly differs from the one of the usual polyatomic invariants through the additional condition $|\eta(R) - \eta(R')| \leq \e$.
Therefore, the set of collisions on which the invariance property is assumed to hold is smaller than in the standard polyatomic case. As a result, we cannot directly apply the known results of characterization of the collision invariants in the generic case. We provide this characterization in the following proposition, whose proof is mostly based on \cite[Lemma 3]{boudin:hal-03629556}, and provided in Appendix~\ref{a:collinv} for the sake of completeness.

\begin{proposition}[Collision invariants] \label{collision_invariants_quasires} 
Assume $\psi$ to be a continuous $\varepsilon$-quasi-resonant collision invariant. Then there exist $a_1 \in \R$, $p \in \R^3$ and $a_2 \in \R$ such that for all $(v,I) \in \R^3 \times \R_+$, we have
\begin{equation} \label{eq:characterization_invariant_quasires}
\psi(v,I) = a_1 + p \cdot v + a_2 \left(\frac{m}{2} |v|^2 + I \right).
\end{equation}
\end{proposition}
The collision invariants \eqref{eq:characterization_invariant_quasires} are here the same as in the usual polyatomic case, rather than those of the resonant case. Having derived the weak form of the Boltzmann collision operator and characterized the collision invariants, we are now in a position to state the $H$ theorem associated with the quasi-resonant model.

\begin{theorem}[$H$ theorem] \label{theorem:H_quasi_res}
For any density $f \equiv f(v,I) > 0$, we have, at least formally, 
\begin{equation} \label{part1thH_quasires}
\int_{\R^3} \int_0^{+\infty} \log f(v,I) \, Q_\e(f,f)(v,I)  \, \dd v \,  \varphi(I) \, \dd I \leq 0.
\end{equation}
Moreover, the following two conditions are equivalent:
\begin{equation} \label{part2thH_quasires1}
\int_{\R^3} \int_0^{+\infty} \log f(v,I) \, Q_\e(f,f)(v,I)  \, \dd v  \, \varphi(I) \, \dd I = 0
\end{equation}
and 
\begin{equation} \label{part2thH_quasires2}
f(v,I) = \M(v,I) := \frac{\rho}{m} \, \left(\frac{2 \pi T}{m}\right)^{-3/2} \, Z(1/T)^{-1} \exp \left( -\frac{m |v-u|^2}{2T} - \frac{I}{T}\right),
\end{equation}
where 
$$Z(\beta) = \int_0^{+\infty}e^{-\beta I} \, \varphi(I) \, \dd I$$
is the internal partition function, $\rho$, $u$ and $T$ are respectively the mass density, average velocity and average temperature associated with $f$, \textnormal{i.e.} satisfy
\begin{equation*}
    \begin{pmatrix} \rho \\ \rho u \\ \dfrac{3 + \delta(T)}{2} \rho T +  \dfrac{1}{2} \rho |u|^2 \end{pmatrix} = \int_{\R^3} \int_0^{+\infty} f(v,I) \, \begin{pmatrix} 1 \\ v \\ \dfrac12 |v|^2 + I \end{pmatrix} \, \dd v \, \varphi(I) \, \dd I,
\end{equation*}
and where the number of internal degrees of freedom $\delta(T)$ is given by
\begin{equation*}
    \delta(T) = - \frac{2}{T} (\log Z)' \left(\frac{1}{T} \right).
\end{equation*}
\end{theorem}
Notice that, in the quasi-resonant case, the equilibrium distributions are, as in the usual polyatomic case, Maxwell functions with a single temperature. We recall that, in the resonant case, these distributions have two distinct temperatures (one kinetic and one internal).

The proof of Theorem~\ref{theorem:H_quasi_res} is straightforward once collision invariants are characterized (here, thanks to Proposition~\ref{collision_invariants_quasires}) and very standard, see for instance the proof of \cite[Theorem 4.1]{borsoni2022general}.

\subsection{Resonant asymptotics} \label{ssec:asymptoticsres_quasires}
In this subsection, we investigate the consistency of the quasi-resonant collision kernels with the corresponding resonant collision kernel in the vanishing $\e$ asymptotics. 
\begin{proposition} \label{p:mesphireson}
Consider $B$ a continuous reference collision kernel, hence satisfying \eqref{e:genBsymm}--\eqref{e:genB>0}.
Let $\e>0$, and $\psi := \psi(I',I'_*)$ be a continuous test-function with a compact support in $(0,+\infty)^2$. The tuple of variables $(v,v_*,I,I_*,\sigma) \in  (\R^3)^2 \times (0,+\infty)^2 \times \Sb^2$ being fixed, we set
\begin{equation*}
    \mathcal{J}_{\e} = \iint_{(0,+\infty)^2} B_{\e}(v,v_*,I,I_*,I',I'_*,\sigma) \, \psi(I',I'_*) \,\varphi(I')\,\varphi(I_*')\, \dd I' \,\dd I_*'.
\end{equation*}
Then $(\mathcal{J}_{\e})$ converges, as $\e$ goes to $0$, towards
\begin{equation} \label{e:valeurJ0}
    \mathcal{J}_0 = \int_0^{+\infty} B^\textnormal{res}(v,v_*,I,I_*,I',\sigma) \, \psi(I',I+I_* - I') \, \varphi(I') \, \varphi(I+I_* - I') \, \dd I',
\end{equation}
where the resonant collision kernel $B^\textnormal{res}$ associated to the reference kernel $B$ is defined by
\begin{equation} \label{eqdef:resonantkernelassociatedtoB_quasires}
    B^\textnormal{res}(v,v_*,I,I_*,I',\sigma) = B(v,v_*,I,I_*,I',I+I_* - I',\sigma) \, E \, \mathds{1}_{[I' \leq I+I_*]},
\end{equation}
\end{proposition}
Proposition~\ref{p:mesphireson} confirms that, from a modelling viewpoint, considering a small $\e$ in our approach indeed allows to describe the quasi-resonance phenomenon. More precisely, it is clear that $B$ will not be the corresponding resonant cross-section, and that it is not enough to impose $I_*'=I+I_*-I'$: the measure on $I'$ is also modified, and an additional term $E$ rises. Hence, in \eqref{e:valeurJ0}, we obtain the resonant integration measure $\varphi(I') \, \varphi(I+I_* - I') \, \dd I'$ used in the corrected version \cite{borsonicompactcorrect} of \cite{borsonicompact}. Moreover, $B^\textnormal{res}$ still complies with the symmetry and micro-reversibility properties. In fact, the characteristic function can be omitted, as the argument $I+I_* - I'$ of $B$ ensures nonnegativity whenever $B$ is non-zero.
\begin{proof}
Since $v$, $v_*$, $I$, $I_*$ and $\sigma$ are fixed, we set, in the following, for the sake of clarity,
\begin{equation*}
    \Psi(I',I'_*) = B(v,v_*,I,I_*,I',I'_*,\sigma) \, \psi(I',I'_*) \, \varphi(I') \, \varphi(I'_*).
\end{equation*}
Using \eqref{e:defchiepsi}--\eqref{e:defBepsi}, we get
$$\mathcal{J}_{\e} = \frac 1{2\e} \iint_{(0,+\infty)^2}c_\eta(R,R')\, \mathds{1}_{\left[|\eta(R) - \eta(R')| \le \e\right]}(R,R') \,  \Psi(I',I'_*) \,  \dd I' \, \dd I'_*.$$
Remembering that $R'$ only depends on some of the fixed variables, and not on $I'$ or $I_*'$, a convenient way to proceed is to use the Borgnakke-Larsen representation, that is to change variables $(I',I'_*)$ into $(R,r)$ with 
\begin{equation*} 
R= 1-\frac{I'+I'_*}{E}, \qquad r=\frac{I'}{I' + I'_*}, 
\end{equation*}
with Jacobian $(1-R) E^2$. Then we have
\begin{equation} \label{e:JepsiRr}
    \mathcal{J}_{\e} = \frac 1{2\e} \iint_{(0,1)^2} c_\eta(R,R')\, \mathds{1}_{\left[|\eta(R) - \eta(R')| \le\e\right]}(R,R') \,  \Psi(r(1-R)E,(1-r)(1-R)E) \, (1-R)E^2\, \dd R \, \dd r. 
\end{equation}
Consider both $r$ and $E$ to be fixed and denote, for any $R$, 
$$\bar{\Psi}(R) = \Psi(r(1-R)E,(1-r)(1-R)E) \; (1-R)E^2,$$
which is clearly continuous by assumption. We now focus, in \eqref{e:JepsiRr}, on the integral in $R$, that is
$$\frac1{2 \e}\int_0^1 c_\eta(R,R') \, \mathds{1}_{\left[|\eta(R) - \eta(R')| \le \e\right]}(R,R') \, \bar{\Psi}(R) \, \dd R.$$
We perform the change of variable $z = \eta(R)$, so that it becomes
\begin{multline} \label{e:limitPsiepsi}
\frac{1}{2 \e} \int_{\R} c_\eta(\eta^{-1}(z),R')  \, \mathds{1}_{\left[|z - \eta(R')| \le \e\right](\eta^{-1}(z),R')} \, \bar{\Psi}(\eta^{-1}(z)) \, (\eta^{-1})'(z) \, \dd z \\
= \frac{1}{2 \e} \int_{\eta(R')-\e}^{\eta(R') + \e} c_\eta(\eta^{-1}(z),R') \, \bar{\Psi}(\eta^{-1}(z)) \, (\eta^{-1})'(z) \, \dd z.
\end{multline}
Letting $\e$ go to $0$ in \eqref{e:limitPsiepsi} and using \eqref{e:reson-c_eta}, we get, at the limit,  
$$c_\eta(R',R') \, \bar{\Psi}(R') \, (\eta^{-1})'(\eta(R')) = \frac{c_\eta(R',R') \,\bar{\Psi}(R')}{\eta'(R')} = \bar{\Psi}(R').$$
Hence, by dominated convergence, $(\mathcal{J}_{\e})$ converges, when $\varepsilon$ goes to $0$, towards
$$\int_0^1 \Psi(r(1-R') E,(1-r)(1-R') E) \, (1-R') E^2\, \dd r,$$
which we denote by $\mathcal{J}_0$. Remembering that $(1-R') E = I + I_*$, we finally perform the change of variable $r$ to $I'$ with  $I' = r (I+I_*)$, leading to
\begin{equation*}
\mathcal{J}_0 = \int_0^{I+I_*} \Psi(I',I+I_* - I') \,E \, \dd I',
\end{equation*}
and then to \eqref{e:valeurJ0}.
\end{proof}

\section{Derivation of Landau--Teller-type equations} \label{sec:2kasires}

In this section, we discuss what we can deduce from the $H$ theorem (Theorem~\ref{theorem:H_quasi_res}) and the resonant asymptotics (Proposition~\ref{ssec:asymptoticsres_quasires}). We consider from now on the \textit{homogeneous quasi-resonant Boltzmann equation}, for some $\e>0$,
\begin{equation} \label{eq:homogeneousBoltzquasires}
    \partial_t f(t,v,I) = Q_\e (f(t),f(t))(v,I), \qquad \qquad f(0,v,I) = f_\textnormal{in}(v,I),
\end{equation}
for $t \geq 0$, $v \in \R^3$ and $I \in \R_+$.

\subsection*{Expected properties of the quasi-resonant dynamics}
Let us first informally discuss the properties of the quasi-resonant model we may expect, based on the results proven earlier. The $H$ theorem ensures that the equilibrium distributions of the {quasi-resonant} dynamics have a Maxwellian form, with one temperature. On the other hand, for the resonant asymptotics, Proposition~\ref{ssec:asymptoticsres_quasires} ensures that, at least formally, when $\e$ is small, the quasi-resonant dynamics should behave almost like the associated resonant ones. But we also recall that the equilibrium distributions of the resonant dynamics are Maxwell functions with two temperatures.

\begin{figure}[ht]
    \centering
    \begin{tikzpicture}[scale=1]
  
\draw[fill=black!10,draw=none] (0,0) rectangle (5,2);
\draw[fill=blue!10!red!10,draw=none] (5,0) rectangle (12,2);

\draw[->,thick] (0,0) -- (12,0);

\draw (12,0) node[below]{time};

\draw[dashed, thin] (5,0) -- (5,2);

\draw  (2.5,1.25) node{convergence towards a};
\draw  (2.5,.75) node{two-temperature Maxwellian};

\draw (8.5,1.25) node{relaxation of the two temperatures};
\draw (8.5,.75) node{towards each other};

\draw  (2.5,-.3) node{short time};

\draw  (8.5,-.3) node{long time};

\end{tikzpicture}
    \caption{Expected behaviour of the quasi-resonant dynamics.}
    \label{fig:rep_quasires_dynamic}
\end{figure}
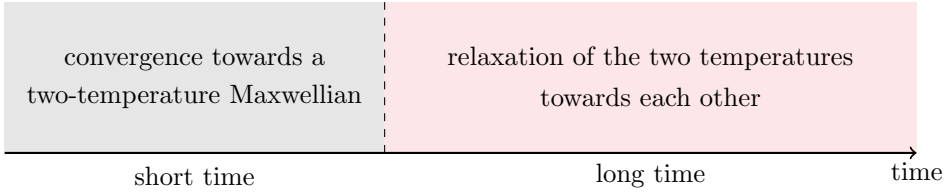

Therefore, we expect the quasi-resonant dynamics to be composed of two phases, as schematically illustrated in Figure~\ref{fig:rep_quasires_dynamic}. In the short term, the quasi-resonant dynamics closely resembles its associated resonant dynamics, leading to a relaxation towards a two-temperature Maxwellian. Over longer time scales, since the distribution remains close to this form, the two temperatures should gradually converge towards each other. Moreover, we anticipate that the distribution will retain the structure of a two-temperature Maxwellian, as the resonant component of the quasi-resonant collision operator dominates its mixing component.

This behaviour differs from the dynamics of the non-resonant polyatomic case and is a specific feature of the present model. In particular, we deduce a pair of coupled ordinary differential equations on both kinetic and internal temperatures, of Landau--Teller type, see for instance \cite{nikitin200870}.




Let us now obtain explicit formulas, at the lowest non-zero order of approximation in $\e$, for the aforementioned two temperatures in the second, long-time, part of the dynamics. To this end, we derive ODEs of Landau--Teller type on the kinetic and internal temperatures, which are locally satisfied at time $t=0$ by the solutions to the quasi-resonant dynamics \eqref{eq:homogeneousBoltzquasires} with a two-temperature Maxwellian initial distribution. The heuristic discussed above on the two phases of the dynamics allows to see those ODEs as good candidates for actually computing the two temperatures at all times: this is later highlighted by numerical experiments presented in Section~\ref{section:numerical_quasires}.

\subsection*{Computational choices}

Before stating the main result of this section, we first recall the expressions we use in the computations below. We emphasize once again that these expressions were not needed in the previous section. First, the energy law density $\varphi$ is chosen as
\begin{equation} \label{choix:varphi_quasires}
\varphi(I)=I^{\delta/2-1}, \qquad I \geq 0,
\end{equation}
with $\delta \geq 2$. 
This choice of energy law density corresponds to a polytropic gas with a number of internal degrees of freedom equal to $\delta$, and is common in the literature, see \cite{alonso2024cauchy,MR1277241}.

Second, the reference collision kernel is set to 
\begin{equation} \label{eq:refkernel_computations}
B(v,v_*,I,I_*,I',I'_*,\sigma) = C_B \; b(\sigma \cdot \sigma') \; E_k^{\kappa_k-\frac12}  \;{E_k'}^{\kappa_k} \; E_i^{\kappa_i}  \; {E_i'}^{\kappa_i} \; E^{\gamma} \; \frac{\mathds{1}_{[I'+I'_* \leq E]} }{\mathfrak{m}_\varphi (E)}, 
\end{equation}
where we use the shortcut notations
$$E_k =\frac{m}{4}|v-v_*|^2, \qquad E_i=I+I_*, \qquad E_k'=\frac{m}{4}|v'-v'_*|^2, \qquad E_i' = I'+I'_*, \qquad E = E_k + E_i = E_k' + E_i',$$ 
$v'$ and $v'_*$ are given by \eqref{e:prepostvit}, $\kappa_k$, $\kappa_i$ and $\gamma$ are real numbers, $C_B$ is a positive constant, and $b$ is an angular kernel satisfying, for $\omega\in\Sb^2$,
\begin{equation} \label{eq:angular_quasires}
\int_{\Sb^2} b(\sigma \cdot \omega) \, \dd \sigma = 1.
\end{equation}
Eventually, taking \eqref{choix:varphi_quasires} into account in \eqref{e:gfunction}, and thanks to Lemma~\ref{l:exprmmphi}, we have
\begin{equation} \label{mvarphi_explicit_quasires} 
\mathfrak{m}_\varphi (E) = \iint_{(0,+\infty)^2} \, \mathds{1}_{[I'+I'_* \leq E]} \, \varphi(I')\, \varphi(I_*')\, \dd I' \, \dd I_*' =\frac{\Gamma(\delta/2)^2}{\Gamma(\delta+1)} \; E^{\delta},
\end{equation}
where $\Gamma$ stands for the Gamma function.

Finally, we choose $\eta$ as
\begin{equation} \label{eta:logR}
    \eta(R) = \log \left(\frac{R}{1-R} \right), \qquad R \in (0,1).
\end{equation}
Let $\e>0$. That function $\eta$ allows to define the truncation function $\chi_\e$ through \eqref{e:defchiepsi}--\eqref{e:reson-c_eta} (with no need to specify $c_\eta$) and subsequently $B_\e$ as 
\begin{equation} \label{e:B_esp_B_ref}
B_{\e}(v,v_*,I,I_*,I',I'_*,\sigma) = B(v,v_*,I,I_*,I',I'_*,\sigma) \, \chi_{\e}(R,R').
\end{equation}

\subsection*{Local Landau--Teller equations} We can now provide, in the next proposition, the first term in the expansion in $\e$ of the derivative at time $t=0$ of the internal temperature of a solution to the homogeneous quasi-resonant Boltzmann equation~\eqref{eq:homogeneousBoltzquasires} with a two-temperature Maxwellian initial condition.
\begin{proposition} \label{prop:Landau--Teller_quasires}
For $\e > 0$, consider $f$ a solution to the homogeneous quasi-resonant Boltzmann equation~\eqref{eq:homogeneousBoltzquasires} associated to $\varphi$, $\eta$ and $B_\e$ defined thanks to \eqref{choix:varphi_quasires}--\eqref{e:B_esp_B_ref}, with $\gamma = \delta + 1$. Assume the initial condition $f_\textnormal{in}$ to be a \textnormal{two-temperature} Maxwellian distribution associated to the energy law $\varphi$, the mass density $\rho > 0$, the average velocity $u \in \R^3$ and the kinetic and internal temperatures $T_k^0 > 0$ and $T_i^0 > 0$. For $t \geq 0$, we set
\begin{equation} \label{eq:kinetic_and_internal_temp_quasires}
    T_k(t) = \frac{m}{6 \rho} \iint_{\R^3 \times (0,+\infty)}\!\! |v-u|^2 \, f(t,v,I) \, \dd v \, \varphi(I) \, \dd I, \quad T_i(t) = \frac{1}{\delta \rho} \iint_{\R^3 \times (0,+\infty)} \!\! I \, f(t,v,I) \, \dd v \, \varphi(I) \, \dd I,
\end{equation}
\textnormal{i.e.} the respective kinetic and internal temperatures of $f$ at time $t$. Then 
\begin{equation*}
    \left. \frac{\dd T_i}{\dd t} \right|_{t=0} =   \e^2 \; \rho \,  C_{\delta,\kappa_k,\kappa_i} \left(T_k^0\right)^{2 \kappa_k + \frac12} \; \left(T_i^0\right)^{2 \kappa_i + \delta} \; \left(T_k^0 - T_i^0\right) +  o(\e^2),
\end{equation*} 
with
\begin{equation*}
C_{\delta,\kappa_k,\kappa_i} =  \frac{C_B}{12} \times \frac{ \Gamma \left(2 \kappa_k + 3 \right) \; \Gamma \left(2 \kappa_i + 2 \delta + 1 \right)}{\Gamma(3/2) \; \Gamma(\delta) }.
\end{equation*}
\end{proposition}

\begin{proof}
Since $f$ solves \eqref{eq:homogeneousBoltzquasires}, we have
\begin{equation} \label{e:triangle} 
\left. \frac{\dd T_i}{\dd t} \right|_{t=0} = \frac{1}{\delta \rho} \iint_{\R^3 \times (0,+\infty)} I  Q_{\e}(f_\textnormal{in}, \,f_\textnormal{in})(v,I) \, \varphi(I) \, \dd I \, \dd v.
\end{equation}
Hence, we need to compute the right-hand side of \eqref{e:triangle}. It comes from the weak form \eqref{eq:weakform} of the operator $Q_\e$ that
\begin{multline*}
\frac{1}{\delta \rho} \iint_{\R^3 \times (0,+\infty)} I Q_{\e}(f_\textnormal{in}, \,f_\textnormal{in})(v,I) \, \varphi(I) \, \dd I\, \dd v \\
=\frac{1}{2 \delta \rho} \iint_{(\R^3)^2} \iiiint_{(0,+\infty)^4} \int_{\mathbb{S}^2}  f_\textnormal{in} \,(f_\textnormal{in})_* \, (I'+I'_* - I - I_*) \phantom{\iint_{(\R^3)^2}} \\ \phantom{\iint_{(\R^3)^2}} B_\e(v,v_*,I,I_*,I',I'_*,\sigma)\,\dd \sigma \,  \varphi(I) \, \varphi(I_*) \, \varphi(I') \, \varphi(I'_*)  \,  \dd I  \, \dd I_*  \, \dd I'  \, \dd I'_*  \, \dd v \, \dd v_*.
\end{multline*}
The latter integral can be rewritten as
\begin{multline} \label{eqprooflandauteller_quasires}
    \frac{1}{2\delta \rho} \iint_{(\R^3)^2} \iint_{(0,+\infty)^2}  f_\textnormal{in} \,(f_\textnormal{in})_* \\
    \left[\iint_{(0,+\infty)^2} \int_{\mathbb{S}^2} (I'+I'_* - I - I_*)  \, B_{\e}(v,v_*,I,I_*,I',I'_*,\sigma)\, \varphi(I') \, \varphi(I'_*)  \,  \dd \sigma \, \dd I'  \, \dd I'_* \right] \\
 \phantom{\iint_{(\R^3)^2}}\varphi(I) \, \varphi(I_*)  \,  \dd I  \, \dd I_*  \, \dd v \, \dd v_*.
\end{multline}
We now focus on the term between brackets in \eqref{eqprooflandauteller_quasires}. Using the explicit form \eqref{eq:refkernel_computations} of the reference collision kernel $B$ as well as its angular property \eqref{eq:angular_quasires}, the term becomes 
\begin{equation*}
C_B\frac{E_k^{\kappa_k - 1/2} E_i^{\kappa_i} E^{\gamma}}{\mathfrak{m}_\varphi (E)}\iint_{(0,+\infty)^2} (E_i'-E_i) \, (E-E_i')^{\kappa_k}  \, {E_i'}^{\kappa_i} \, \mathds{1}_{[E_i' \leq E]} \,  \chi_{\e}\left(1-\frac{E_i'}{E},1-\frac{E_i}{E}\right) \, \varphi(I')\, \dd I'  \, \varphi(I'_*) \,\dd I'_*.
\end{equation*}
By definition \eqref{e:gfunction} of $\mathfrak{m}_{\varphi}$, substituting $E_i' = I'+I'_*$ as an integration variable, the above term reads
\begin{equation*}
C_B\frac{E_k^{\kappa_k - 1/2}E_i^{\kappa_i} E^{\gamma}}{\mathfrak{m}_\varphi (E)}\int_{0}^E (E_i'-E_i) \, (E-E_i')^{\kappa_k}  \, {E_i'}^{\kappa_i} \, \chi_{\e}\left(1-\frac{E_i'}{E},1-\frac{E_i}{E}\right) \, \mathfrak{m}'_\varphi (E_i') \, \dd E_i'.
\end{equation*}
Then, using the explicit form \eqref{mvarphi_explicit_quasires} of $\mathfrak{m}_{\varphi}$, performing the change of variable $R = 1 - E_i'/E$ and setting $R_0 = 1 - E_i/E$, we transform the integral into
\begin{equation} \label{e:carre}
C_B \delta\,  E_k^{\kappa_k - 1/2} E_i^{\kappa_i} E^{\gamma + \kappa_k + \kappa_i + 1}\int_{0}^1 (R_0-R) \, R^{\kappa_k}  \; (1-R)^{\kappa_i + \delta - 1} \, \chi_{\e}\left(R,R_0\right) \, \dd R.
\end{equation}
Applying Lemma~\ref{lemma:developpementlimite} to $\psi(R) = - R^{\kappa_k} (1-R)^{\kappa_i + \delta - 1}$ yields the expansion of \eqref{e:carre} with respect to $\e$ around $0$ as
\begin{equation*}
\e^2 \, \frac{C_B (\delta + 1)\,  E_k^{\kappa_k - 1/2} E_i^{\kappa_i} E^{\gamma + \kappa_k + \kappa_i + 1}\,{R_0}^{\kappa_k} (1-R_0)^{\kappa_i + \delta - 1}}{3 \eta'(R_0)^2}\left[ \left(\log \eta' \right)'(R_0) - \frac{\kappa_k}{R_0} + \frac{\kappa_i + \delta - 1}{1 - R_0} \right] + o(\e^2).   
\end{equation*}
Now, since we chose $\eta(R_0) = \log \frac{R_0}{1-R_0}$, we have $\left(\log \eta' \right)'(R_0) = -\frac{1}{R_0} + \frac{1}{1-R_0}$ and $\frac{1}{\eta'(R_0)^2} = R_0^2(1-R_0)^2$, the expression actually equals
\begin{equation*}
\e^2 \, \frac{C_B\delta}{3}  \, E_k^{\kappa_k - 1/2} E_i^{\kappa_i} E^{\gamma + \kappa_k + \kappa_i + 1}\,R_0^{\kappa_k + 2} (1-R_0)^{\kappa_i + \delta + 1} \left(- \frac{\kappa_k + 1}{R_0} + \frac{\kappa_i + \delta}{1 - R_0} \right) + o(\e^2).   
\end{equation*}
As $R_0 E = E_k$ and $(1-R_0)E = E_i$, the previous term becomes
\begin{equation*}
\e^2 \, \frac{C_B\delta}{3}  \, E_k^{2\kappa_k + \frac12} E_i^{2\kappa_i + \delta} E^{\gamma -\delta-1} \left[(\kappa_i + \delta)E_k - (\kappa_k + 1)E_i \right] + o(\e^2).   
\end{equation*}
Finally, recalling that $\gamma = \delta + 1$, the bracket term in \eqref{eqprooflandauteller_quasires} reads
\begin{equation}
\e^2 \, \frac{C_B\delta}{3} \, E_k^{2\kappa_k + 1/2} E_i^{2\kappa_i + \delta}  \left[(\kappa_i + \delta)E_k - (\kappa_k + 1)E_i \right] + o(\e^2).   
\end{equation}
Therefore, \eqref{e:triangle} can be rewritten, when $\e$ goes to $0$, as
\begin{multline*}
 \left. \frac{\dd T_i}{\dd t} \right|_{t=0} = \e^2 \; \frac{C_B}{6\rho}  \iint_{(\R^3)^2} \iint_{(0,+\infty)^2}  E_k^{2\kappa_k + 1/2}  E_i^{2\kappa_i + \delta}  \left[(\kappa_i + \delta) E_k - (\kappa_k + 1) E_i\right] \\
    \phantom{\iint_{(0,+\infty)^2}}\times f_\textnormal{in} \,(f_\textnormal{in})_*  \; I^{\delta/2-1} \dd I  \, {I_*}^{\delta/2-1} \dd I_*  \, \dd v \, \dd v_*  + o(\e^2).
\end{multline*}
Recall that $f_\textnormal{in}$ is a two-temperature Maxwell function with mass density $\rho$, and respective kinetic and internal temperatures $T_k^0$ and $T_i^0$. The previous equation may be recast as
\begin{multline} \label{eqproof2landauteller_quasires}
     \left. \frac{\dd T_i}{\dd t} \right|_{t=0} =\e^2 \, \frac{C_B\rho}{6} \left[(\kappa_i + \delta)\mathbf{m}_k^{2\kappa_k + 3/2}(T_k^0) \, \mathbf{m}_i^{2\kappa_i + \delta}(T_i^0) \right. \\
     \phantom{\frac{C_B\rho}{6}} \left. - (\kappa_k + 1) \mathbf{m}_k^{2\kappa_k + 1/2}(T_k^0) \, \mathbf{m}_i^{2\kappa_i + \delta + 1}(T_i^0) \right] + o(\e^2),
\end{multline}
where we set, for any $\beta>0$ and $T>0$, 
\begin{align}
   \mathbf{m}_k^{\beta}(T) &=  \left(\frac{2 \pi T}{m}\right)^{-3} \iint_{(\R^3)^2} e^{-m(|v|^2 +|v_*|^2)/2T} \left[\frac{m}{4}|v-v_*|^2 \right]^{\beta} \, \dd v \, \dd v_*,  \label{eq:momentbetav_core}\\
   \mathbf{m}_{i}^{\beta}(T) &=  \left[\iint_{(0,+\infty)^2}\!\!\!\! e^{-(I+I_*)/T} \, (I I_*)^{\delta/2 - 1} \, \dd I \, \dd I_* \right]^{-1}\!\! \iint_{(0,+\infty)^2}\!\!\!\! e^{-(I+I_*)/T} \, (I+I_*)^{\beta} \, (I I_*)^{\delta/2 - 1} \, \dd I \, \dd I_*. \label{eq:momentbetai_core}
\end{align}
We point out that the latter depends on $\delta$ too. Then we plug \eqref{eq:momentbetav}--\eqref{eq:momentbetai} from Lemma~\ref{lemma:moments} in \eqref{eqproof2landauteller_quasires} to get
\begin{equation*} 
    \left. \frac{\dd T_i}{\dd t} \right|_{t=0} =\e^2 \, \frac{C_B\rho}{12} \, \frac{\Gamma(2 \kappa_k + 3) \Gamma(2 \kappa_i + 2\delta + 1)}{\Gamma(3/2) \Gamma(\delta)} \, (T_k^0)^{2\kappa_k + 1/2} (T_i^0)^{2\kappa_i + \delta} \left(T_k^0 - T_i^0 \right) + o(\e^2),
\end{equation*}
which ends the proof.
\end{proof}

\subsection*{Global Landau--Teller equations in the quasi-resonant case} 
In the previous Proposition~\ref{prop:Landau--Teller_quasires}, we obtained an explicit formula, for small $\e$, of the derivative around time $t=0$ of the internal temperature of a solution to the homogeneous quasi-resonant Boltzmann equation~\eqref{eq:homogeneousBoltzquasires} with a two-temperature Maxwellian initial condition. Following our discussion from the beginning of the current section, a solution to the homogeneous quasi-resonant Boltzmann equation is expected to remain close to the two-temperature Maxwellian form. 

Hence, one can extrapolate that, if $f$ solves \eqref{eq:homogeneousBoltzquasires} with a kernel $B_\e$ defined in~\eqref{e:B_esp_B_ref}, for a small $\e$, and with an initial condition close to a two-temperature Maxwellian, we have for any time $t \geq 0$,
\begin{equation} \label{eq:approx}
    (T_i(t), \, T_k(t)) \approx (\overline{T_i}(t), \, \overline{T_k}(t)),
\end{equation}
where $(\overline{T_i}, \, \overline{T_k})$ solves the ODE system
\begin{align}
       &\frac{\dd\overline{T}_i}{\dd t}(t) =  \e^2 \; \rho \; C_{\delta,\kappa_k,\kappa_i} \; \;  \overline{T}_k(t)^{2 \kappa_k + 1/2} \; \;  \overline{T}_i(t)^{2 \kappa_i + \delta} \; \; (\overline{T}_k(t) - \overline{T}_i(t)), \qquad t \geq 0, \label{ode1}  \\
      &\frac32  \frac{\dd\overline{T}_k}{\dd t}(t) + \frac{\delta}{2} \frac{\dd\overline{T}_i}{\dd t}(t) = 0, \quad \qquad \qquad \qquad \qquad \qquad \qquad \qquad \qquad \quad \; \; t \geq 0,\\
      &(\overline{T}_i(0),\overline{T}_k(0)) = (T_i(0),T_k(0)), \label{ode3}
\end{align}
using the same notations as in Proposition~\ref{prop:Landau--Teller_quasires}. The goal of the next section is then to check, through some numerical experiments, the validity of \eqref{eq:approx}.
 
\section{Numerical experiments} \label{section:numerical_quasires}
In this section, we conduct a numerical experiment to check the validity of our statement~\eqref{eq:approx}, providing an equation on the relaxation of the kinetic and internal temperatures towards each other, in a quasi-resonant context. We highlight that it relies on the assumption that, in the quasi-resonant setting, the distribution stays at all times close to a two-temperature Maxwellian, which statement is therefore also numerically verified here.

\medskip

\noindent In order to achieve this goal, we proceed as follows.
\begin{enumerate}[label=(\roman*)]
    \item We simulate the $3$D space-homogeneous Boltzmann equation~\eqref{eq:homogeneousBoltzquasires} associated to the collision kernel~\eqref{e:B_esp_B_ref}, with a two-temperature Maxwellian initial condition. We denote the kinetic and internal temperatures of the solution at time $t$ by $T_k(t)$ and $T_i(t)$.
    \item We numerically solve the ODE system \eqref{ode1}--\eqref{ode3} with the same initial conditions of temperature as the ones of DSMC. 
    \item We compare $t \mapsto T_i(t)$ with $t \mapsto \overline{T}_i(t)$: if the two curves globally coincide, then we observe experimentally the validity of~\eqref{eq:approx} and the expected behaviour (Figure~\ref{fig:rep_quasires_dynamic}).
\end{enumerate}

\subsection{Main experiment: a comparison between Landau-Teller and DSMC} \label{ssec:experiment}

\noindent Numerically solving~\eqref{ode1}--\eqref{ode3} is straightforward. We simulate the three-dimensional homogeneous Boltzmann equation~\eqref{eq:homogeneousBoltzquasires} associated with the collision kernel~\eqref{e:B_esp_B_ref}, with a two-temperature Maxwellian initial condition, through the DSMC method~\cite{bird2013}, where each particle is endowed with its velocity $v \in \R^3$ and internal energy quantile $q\in \R_+$. For a discussion on the use of energy quantiles $q$ instead of energy levels $I$ in numerical simulations, see~\cite{bisi2025modelling}. The mass density, average velocity, average temperature, kinetic and internal temperatures of $f(t)$ are respectively denoted $\rho > 0$, $u \in \R^3$, $T_{\textnormal{eq}} > 0$ (these quantities are constant w.r.t. time) and $T_k(t)$ and $T_i(t)$. In particular, for all $t \geq 0$,
\[
3 T_k(t) + \delta T_i(t) = (3 + \delta) T_{\textnormal{eq}}.
\]
In our DSMC simulation, we consider the functions $\varphi$, $\eta$ and for some $\e > 0$, the kernel $B_\e$ of the form~\eqref{choix:varphi_quasires}--\eqref{e:B_esp_B_ref}, with
\begin{align*}
&C_B = 2, \quad b = \frac{1}{4 \pi}, \quad m = 1, \quad \rho = 1, \quad \delta = 2, \quad  \kappa_k = \frac12, \quad \kappa_i = -\frac12, \quad \gamma = \delta + 1,\\
&c_\eta(R,R') = \sqrt{\eta'(R) \eta'(R')}.
\end{align*}
We set the initial temperatures to be
\[
T_k^0 = 1, \qquad T_i^0 = 50.
\]
We point out that the equilibrium temperature is
\(T_{\textnormal{eq}} = 20.6\) and that all values are here adimensionned. We also mention that there are no restrictions (for instance, of closeness) on the initial temperatures to consider.

\noindent Our main numerical experiment is conducted taking $\e = 10^{-1}$ (which is already in the quasi-resonant regime) and $10^5$ numerical particles. We show in Figure~\ref{fig:xp_num_quasires} the result of the simulation. The black curve of Figure~\ref{fig:xp_num_quasiresA} indicates the time evolution of the internal temperature $T_i$ of the system simulated through the DSMC algorithm, while the gray dashed curve is obtained by solving the ODE system~\eqref{ode1}--\eqref{ode3} with the solver \verb|solve_ivp| from the Python package \verb|scipy.integrate|. The gray dotted line corresponds to the equilibrium temperature $T_{\textnormal{eq}}$. We plot on Figure~\ref{fig:relative_error} the graph of $|T_i - \overline{T}_i|/\overline{T}_i$ in semi-log scale, the relative error between $T_i$ (DSMC) and $\overline{T}_i$ (Landau--Teller), with respect to time.

\begin{figure}[!ht]
    \centering
    \begin{subfigure}[t]{0.7\textwidth}
    \centering
    \includegraphics[width=.85\textwidth]{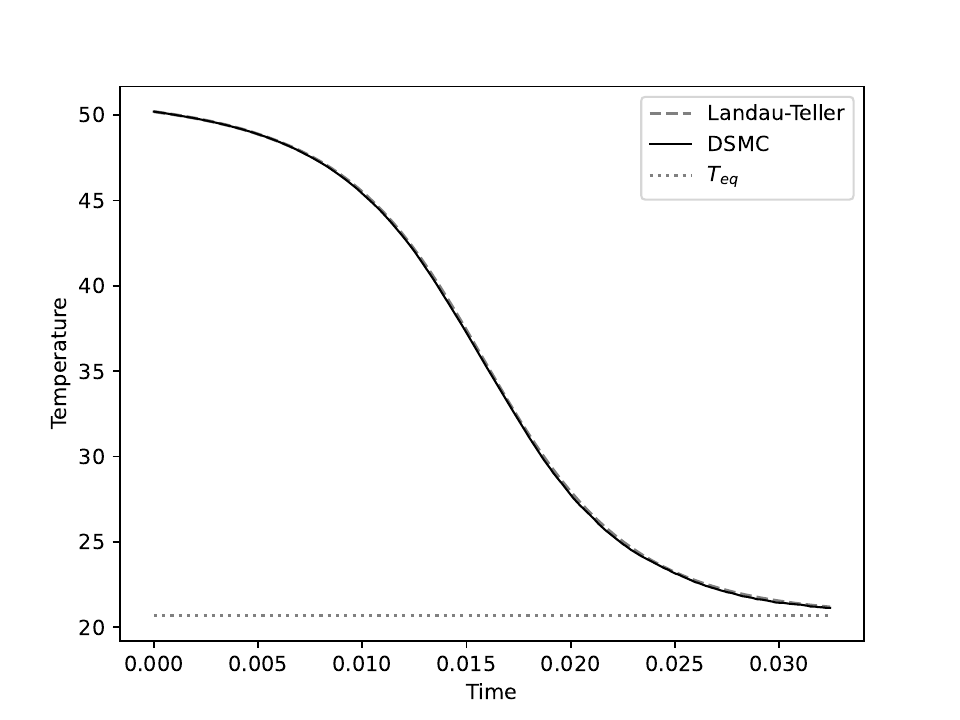}
    \caption{Internal temperatures $T_i$ (DSMC) \\and $\overline{T}_i$ (Landau--Teller) w.r.t. time.}
    \label{fig:xp_num_quasiresA}
    \end{subfigure}
    
    \begin{subfigure}[b]{0.7\textwidth}
            \centering
    \includegraphics[width=.85\textwidth]{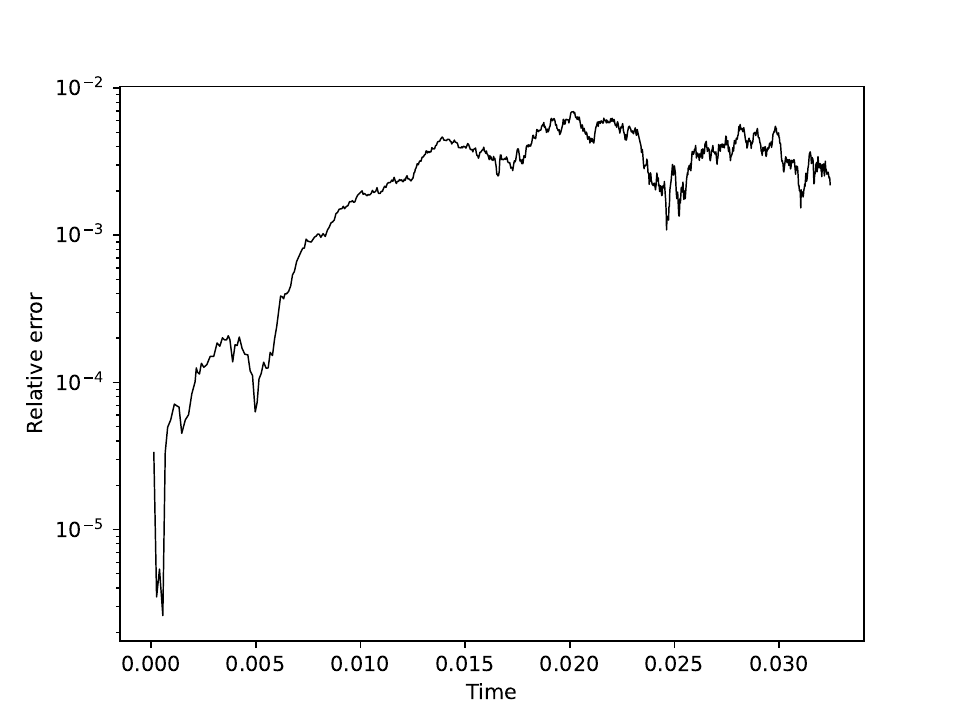}
    \caption{Relative error $|T_i - \overline{T}_i|/\overline{T}_i$ in semi-log scale, \\w.r.t. time.}
    \label{fig:relative_error}
    \end{subfigure}
    \caption{Comparison of the internal temperatures $T_i$ and $\overline{T}_i$ obtained respectively by DSMC simulation and by solving the Landau--Teller ODE system, with $\e = 10^{-1}$.}
    \label{fig:xp_num_quasires}
\end{figure}

\medskip

\noindent We observe on Figure~\ref{fig:xp_num_quasires} a firm agreement between the DSMC-computed $T_i$ and the Landau--Teller-solved $\overline{T}_i$ in a quasi-resonant setting with $\e = 10^{-1}$.

\begin{remark} 
We also conducted a numerical experiment in a far-from-resonance setting, with $\e = 10$, and observed in this case, as expected, no matching between the curves of the internal temperatures given respectively by the DSMC simulation and the solving of the Landau--Teller ODE system.
\end{remark}

\begin{remark} \label{remark:3}
We also conducted two experiments with initial distributions which are not two-temperature Maxwellians: a uniform distribution in both velocity and internal energy, and the product of an anisotropic distribution in velocities and a uniform one in internal energy. Using the Henze-Zikler~\cite{henze1990class} normality and Levene~\cite{brown1974robust} isotropy tests for the velocity part and Kolmogorov-Smirnov test (comparing the quantile functions) for the internal part, we observed a relaxation time toward the two-temperature Maxwellian of the order of $10^{-4}$, much smaller than the relaxation time of the temperatures towards each other, of the order of $10^{-2}$ (see Figure~\ref{fig:xp_num_quasires}), providing a numerical validation of the expected behaviour presented in Figure~\ref{fig:rep_quasires_dynamic}.
\end{remark}

These result support, at least for our chosen set of parameters, the statements of Section~\ref{sec:2kasires}: the kinetic and internal temperatures of a solution to the homogeneous quasi-resonant Boltzmann equation, associated to the kernel $B_\e$ defined in \eqref{choix:varphi_quasires}, for some $\e$ small and two-temperature Maxwellian initial condition, are close to the solution to the Landau--Teller ODE system \eqref{ode1}--\eqref{ode3}.

This moreover somehow provides a numerical validation on the expected behaviour of the quasi-resonant dynamics schematically presented in Figure~\ref{fig:rep_quasires_dynamic}: it does appear reasonable to understand the latter as composed of two parts, a resonant-like dynamics in short-time making the distribution relax towards a two-temperature Maxwellian, and then a relaxation of the two temperatures towards each other driven by a Landau--Teller-type ODE system, in particular in the light of Remark~\ref{remark:3}.

\subsection{Behaviour with a vanishing quasi-resonance parameter} \label{ssec:behaviour_eps_small} In the previous subsection, we checked the validity of our statement \eqref{eq:approx} on a given set of parameters, validity that should be reinforced as the quasi-resonance parameter $\varepsilon$ vanishes. In this subsection, we numerically investigate deeper the behaviour of the relative error between the internal temperatures given respectively by the DSMC simulation and the Landau--Teller system, $|T_i - \overline{T}_i|/\overline{T}_i$, in the vanishing $\e$ limit.

We consider the same set of physical parameters as in the previous subsection. We run the simulation for a range of values of $\e$, and denote for each $\e > 0$ by $t \mapsto T^\e_i(t)$ and $t \mapsto \overline{T}_i^\e(t)$ the internal temperatures obtained respectively with the DSMC simulation and solving the ODE system. We study the error between the dynamics provided by the DSMC simulation and the Landau--Teller ODE system on a complete relaxation to equilibrium. Since the time of relaxation to equilibrium depends itself on $\e$ (it scales as $\e^{-2}$), we consider here, as our measure for discrepancy between the two dynamics, the average relative $L^2$-error between $T_i^{\e}$ and $\overline{T}_i^{\e}$, given by
\[
\left(\frac{1}{\tau_\e}\int_0^{\tau_\e} \left(\frac{|T^\e_i(t) - \overline{T}^\e_i(t)|}{\overline{T}^\e_i(t)} \right)^2 \, \dd t \right)^{\frac12},
\]
where $\tau_\e$ is the time of relaxation to equilibrium, depending on $\e$. We then fit the average relative $L^2$-error, seen as a function of $\e$, to a power law in $\e$, using a linear regression through least squares on the logarithm of the relative $L^2$-error.

We show on Figure~\ref{fig:behaviour_eps} the graph of the relative $L^2$-error as a function of $\e$, obtained by running simulations with over $3 \times 10^5$ particles, for $16$ values of $\e$, ranging from $0.2$ to $0.5$ uniformly in log-scale.

\begin{figure}[!ht]
    \centering
    \begin{subfigure}[t]{0.7\textwidth}
    \centering
    \includegraphics[width=.92\textwidth]{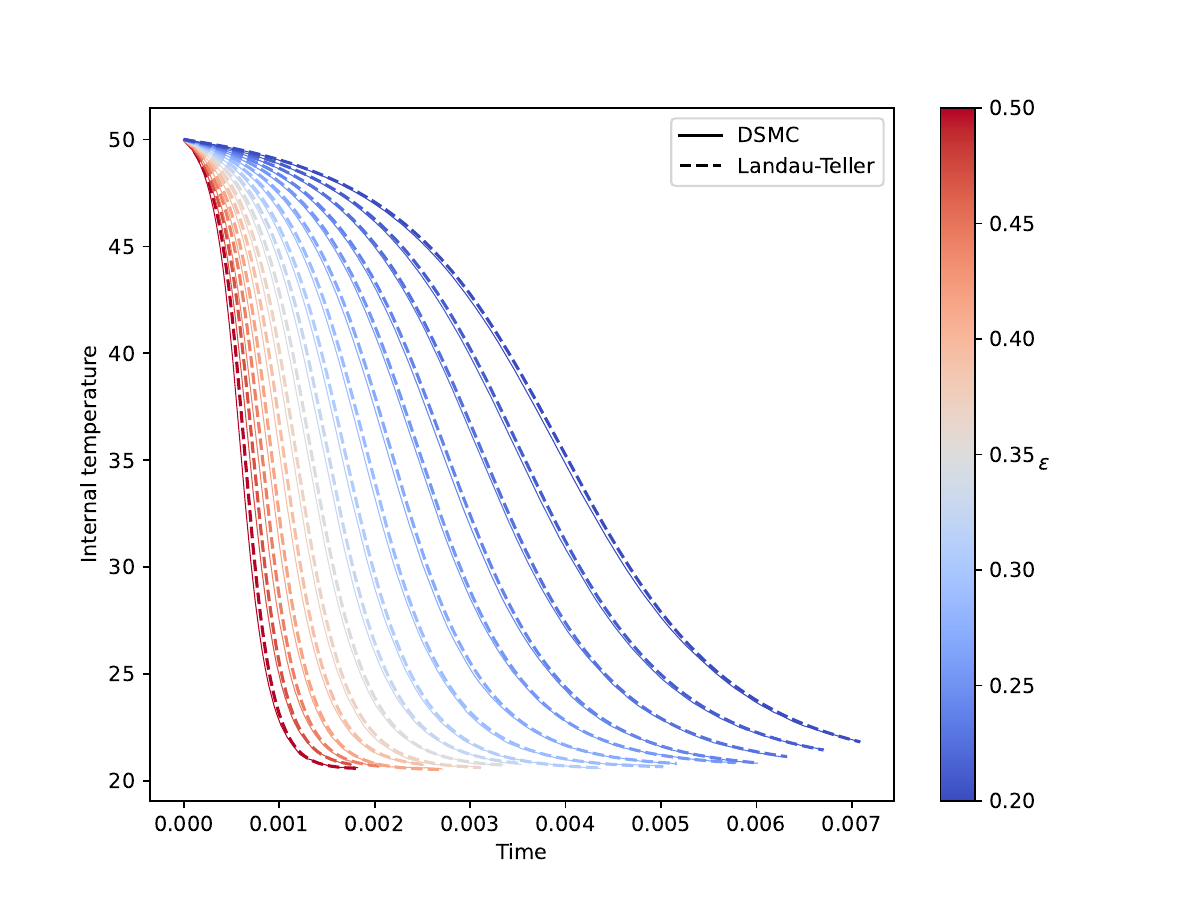}
    \caption{Plots of $t \mapsto T_i(t)$ (DSMC) and $t \mapsto \overline{T}_i(t)$ for various values of $\e$.}
        \label{fig:behaviour_eps_temp}
    \end{subfigure}
    
    \begin{subfigure}[b]{0.7\textwidth}
            \centering
    \includegraphics[width=.83\textwidth]{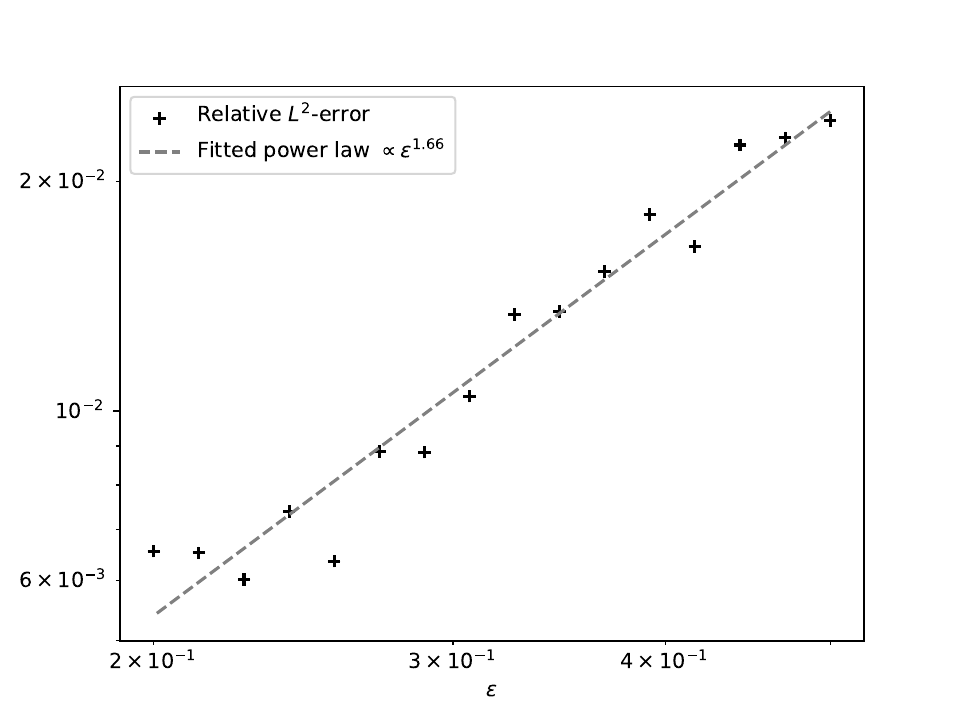}
    \caption{Average relative $L^2$-error between $T_i$ and $\overline{T}_i$ as a function of $\e$.}
    \label{fig:behaviour_eps_err}
    \end{subfigure}
    \caption{Result of the numerical experiment to study the behavior of the average $L^2$-error between $T_i$ and $\overline{T}_i$ relatively to $\e$.}
    \label{fig:behaviour_eps}
\end{figure}

We observe on this test that the average relative $L^2$-error behaves like $\e^{5/3}$. Although $5/3$ is an interesting non-trivial order of convergence, we draw no definite conclusion as of the actual order of convergence, since the comparison is made here with a DSMC simulation, and we also do not exclude that the power may change for smaller values of $\e$.

\subsection{Conclusion} We have numerically observed all behaviours of the quasi-resonant dynamics announced in Section~\ref{sec:2kasires}. Namely, with our parameters and $\e=0.1$, the typical time of convergence to a two-temperature Maxwellian is of the order of $10^{-4}$ (Remark~\ref{remark:3}), which is much shorter than the typical time of convergence of the two temperatures towards each other, of the order of $10^{-2}$ (Figure~\ref{fig:xp_num_quasires}). This latter time indeed becomes larger as $\e$ diminishes (Subfigure~\ref{fig:behaviour_eps_temp}). We have also shown the validity of the Landau--Teller equation (Figure~\ref{fig:xp_num_quasires}), which gets all the more accurate as $\e$ diminishes (Subfigure~\ref{fig:behaviour_eps_err}).

\appendix

\section{Characterization of the collision invariants in the quasi-resonant case} \label{a:collinv}

In this appendix section, we prove that we recover the same collision invariants in the quasi-resonant and standard polyatomic cases.

\begin{proposition} \label{p:collinvQR}
If $\psi$ is a continuous collision invariant in the quasi-resonant case, then there exist $a_1$, $a_2 \in \R$ and $p \in \R^3$ such that, for any $(v,I) \in \R^3 \times (0,+\infty)$,
$$\psi(v,I) = a_1 + p \cdot v + a_2 \left(\frac{1}{2} m|v|^2 + I \right).$$
\end{proposition}

\begin{proof}
As a collision invariant in the quasi-resonant case, $\psi$ satisfies
\begin{equation} \label{e:invarpsi}
\psi \left( \frac{v + v_*}{2} +
2 \sqrt{\frac{E-I'-I_*'}{m}}\sigma , I' \right) + \psi \left( \frac{v + v_*}{2} -
2 \sqrt{\frac{E-I'-I_*'}{m}}\sigma, I_*'\right) = \psi(v,I) + \psi(v_*,I_*)
\end{equation}
for any $(v,v_*,I,I_*,I',I_*',\sigma)$ such that $B_{\e}(v,v_*,I,I_*,I',I_*',\sigma)$ is non-zero. In particular, \eqref{e:invarpsi} holds whenever $R=R'$, using the notations defined in \eqref{e:defRRprime}, \ie~in the resonant case. 
Thanks to the computations detailed in \cite{boudin:hal-03629556}, there exist $a_1$, $a_2$, $a_3\in \R$ and $p \in \R^3$ such that, for any $(v,I)$, 
$$\psi(v,I) = a_1 + p \cdot v + \frac{a_2}{2} m|v|^2 + a_3 I.$$
Replacing $\psi$ by its previous expression in \eqref{e:invarpsi}, we obtain
$$p \cdot (v'+v_*') + \frac{a_2}{2} \left( m|v'|^2 + m|v_*'|^2 \right) +a_3( I' + I_*') = p \cdot (v+v_*) + \frac{a_2}{2} \left( m|v|^2 + m|v_*|^2 \right) +a_3( I + I_*).$$
Thanks to the microscopic conservations \eqref{e:momconserv}--\eqref{e:totalenrgconserv} of momentum and total energy, we get
$$(a_2-a_3) (I' + I'_* - I - I_*) = 0.$$
Finally, we can choose $(v,v_*,I,I_*,I',I_*',\sigma)$ such that $B_{\e}(v,v_*,I,I_*,I',I_*',\sigma)\neq 0$ and $R\neq R'$ simultaneously, so that 
$I'+I'_* \neq I+I_*$, and subsequently obtain $a_2 = a_3$.
\end{proof}

\section{Various technical lemmas}  \label{s:computlemma}

The first lemma is dedicated to the computation of $\mph$ and its derivative, it is presented here for the sake of completeness, as it is a well-known result. 
\begin{lemma} \label{l:exprmmphi}
Recall the expression \eqref{e:gfunction} of $\mph$, that is, for any $E\ge0$, 
\begin{equation} \label{e:gfunctionappen}
    \mph (E)= \iint_{(0,+\infty)^2} \, \mathds{1}_{[I'+I'_* \le E]}(I',I_*')\, \varphi(I')\, \varphi(I_*')\, \dd I' \, \dd I_*'=\int_0^E\int_0^J \varphi(I')\, \varphi(J-I')\, \dd I' \, \dd J.
\end{equation}
Then, for any $E\ge0$, we have
\begin{equation} \label{e:calcmphiprime}
    \mph' (E)= \int_0^E \varphi(I')\, \varphi(E-I')\, \dd I'.
\end{equation}
Moreover, if $\varphi : I \mapsto I^{\delta/2-1}$ for some $\delta>0$, then, for any $E\ge0$, we have
\begin{equation} \label{e:calcmphi}
    \mph'(E) = \frac{\Gamma(\delta/2)^2}{\Gamma(\delta)} \, E^{\delta -1}, \qquad \mph (E) =  \frac{\Gamma(\delta/2)^2}{\Gamma(\delta+1)} \, E^{\delta}.
\end{equation}
\end{lemma}
\begin{proof}
First, \eqref{e:calcmphiprime} is a straightforward consequence of  \eqref{e:gfunctionappen}. Then, if $\varphi : I \mapsto I^{\delta/2-1}$, performing the change of variable $I'\to I'/E$ in \eqref{e:calcmphiprime}, and denoting the Beta function by $\B$, we get, for any $E\ge 0$, 
$$\mph'(E)= \B(\delta/2, \delta/2) E^{\delta -1},$$
and eventually \eqref{e:calcmphi} thanks to the relationship between the Beta and Gamma functions.
\end{proof}

Then the following lemma is an intermediate result leading to relevant Taylor expansions with respect to the quasi-resonance parameter $\e$.

\begin{lemma} \label{lemma:developpementlimite}
Let $\psi \in \mathcal{C}^1((0,1))$. Then, for any $R_0 \in (0,1)$, we have
\begin{equation} \label{e:developpementlimitelemme}
    \int_0^1 (R - R_0) \, \psi(R) \, \chi_{\e}(R,R_0) \, \dd R \underset{\e \to 0}{=} \frac{1}{3 \,
    \eta'(R_0)^3} \left[\eta'(R_0)\psi'(R_0) - \eta''(R_0)\, \psi(R_0) \right] \e^2 + o(\e^2),
\end{equation}
where $\chi_\e$ is defined in \eqref{e:defchiepsi}.
\end{lemma}
\begin{proof}
Let $R_0 \in (0,1)$, denote by $\Psi$ an antiderivative of $R \mapsto (R - R_0) \psi(R) c_{\eta}(R,R_0)$, and set, for any $x\in\R$,
$$g(x) = \eta^{-1} \left( \eta(R_0) + x \right).$$
That latter function $g$ is $\mathcal{C}^\infty$ on $\R$, as $\eta$ is itself $\mathcal{C}^\infty$ and $\eta' > 0$. We may write, for any $\e>0$,
$$\int_0^1 (R - R_0) \, \psi(R) \, \chi_{\e}(R,R_0) \, \dd R = \frac{1}{2\e} \int_{g(-\e)}^{g(\e)} (R - R_0) \, \psi(R) \, c_{\eta}(R,R_0) \, \dd R.$$
We highlight that, while $\Psi$ is only $\mathcal{C}^2$ on $(0,1)$, it is $\mathcal{C}^3$ at the point $R = R_0$. Then a straightforward Taylor expansion of $\Psi \circ g$ near $0$ yields
\begin{equation} \label{e:dlintermed}
\int_0^1 (R - R_0) \, \psi(R) \, \chi_{\e}(R,R_0) \, \dd R = \left(\Psi \circ g \right)'(0) + \frac{\e^2}{6} \, \left(\Psi \circ g \right)'''(0) + o(\e^2).
\end{equation}
The zeroth-order (in $\e$) term in \eqref{e:dlintermed} vanishes since $\left(\Psi \circ g \right)'(0) = \Psi'(R_0)g'(0)=0$. The second-order term involves
$$\left(\Psi \circ g \right)'''(0) = \Psi' (R_0)\,g'''(0) + 3 \Psi''(R_0)\, g'(0)\,g''(0) + \Psi'''(R_0)\, g'(0)^3.$$
By direct computations, using in particular \eqref{e:reson-c_eta}--\eqref{e:deriv-c_eta}, we get
$$\Psi' (R_0) = 0, \qquad \Psi''(R_0) =  \psi(R_0) \, \eta'(R_0), \qquad \Psi'''(R_0) = 2 \psi'(R_0) \, \eta'(R_0) + \psi(R_0) \, \eta''(R_0).$$
Moreover, we also have, remembering that $\eta'>0$,
$$g'(0) = \frac{1}{\eta'(R_0)}, \qquad g''(0) =  -\frac{\eta''(R_0)}{\eta'(R_0)^3}.$$
Putting everything together, we get~\eqref{e:developpementlimitelemme}. 
\end{proof}
\begin{remark} It is worth noticing that, if the regularity of $\psi$ is $\mathcal{C}^2$, the correction term $o(\e^2)$ in \eqref{e:developpementlimitelemme} is actually $\mathcal{O}(\e^4)$. In particular, there is no term of order $\e^3$.
\end{remark}

Eventually, we need the following lemma to obtain our Landau--Teller form with explicit coefficients. 
\begin{lemma} \label{lemma:moments}
Consider the functions $\mathbf{m}_k^\beta$ and $\mathbf{m}_i^\beta$ defined in \eqref{eq:momentbetav_core}--\eqref{eq:momentbetai_core}.
\begin{itemize}
    \item Assume $\beta>-3$. Then, for any $T>0$, we have
    \begin{equation} \label{eq:momentbetav}
    \mathbf{m}_k^{\beta}(T)= \frac{\Gamma (\beta + 3/2)}{\Gamma(3/2)} T^{\beta}.
    \end{equation}
    \item Let $\delta>0$ and assume that $\beta>-\delta$. Then, for any $T>0$, we have
    \begin{equation} \label{eq:momentbetai}
    \mathbf{m}_{i}^{\beta}(T) = \frac{\Gamma(\beta + \delta)}{\Gamma(\delta)} T^{\beta}.
   \end{equation}
\end{itemize}
\end{lemma}
\begin{proof}
We start by proving~\eqref{eq:momentbetav}. We perform the change of variables $(v,v_*) \mapsto (g,G) = (v-v_*, \frac{v+v_*}{2})$ in \eqref{eq:momentbetav_core} and get
\begin{equation*}
\mathbf{m}_k^{\beta}(T) = \left(\frac{4 \pi T}{m}\right)^{-3/2} \int_{\R^3} e^{-m|g|^2/4T} \left[\frac{m}{4}|g|^2 \right]^{\beta} \, \dd g.
\end{equation*}
Changing to polar coordinates leads to
\begin{equation*}
\mathbf{m}_k^{\beta}(T) =\left(\frac{4 \pi T}{m}\right)^{-3/2} \, |\Sb^2| \int_{0}^{+\infty} e^{-mr^2/4T} \left(\frac{mr^2}{4} \right)^{\beta} \, r^2 \, \dd r.
\end{equation*}
Then, performing the change of variable $r \mapsto x = m r^2/4 T$, we obtain 
\begin{equation*}
   \mathbf{m}_k^{\beta}(T)= \left(\frac{4 \pi T}{m}\right)^{-3/2} \, |\Sb^2| \int_{0}^{+\infty} e^{-x} (x T)^{\beta} \, \left( \frac{4 T}{m} \right) x \sqrt{\frac{T}{m}}\, \frac{\dd x}{\sqrt{x}} = \frac{2}{\sqrt{\pi}} \, T^{\beta} \int_0^{+\infty} e^{-x} x^{\beta + 1/2} \, \dd x,
\end{equation*}
which allows to recover \eqref{eq:momentbetav}.

\medskip

To deal with \eqref{eq:momentbetai}, it is enough to compute the following integral, which depends on both $\delta$ and $\beta$,
$$\mathcal{I}_\beta = \iint_{(0,+\infty)^2} e^{-(I+I_*)/T} \, (I+I_*)^{\beta} \, (I I_*)^{\delta/2- 1} \, \dd I \, \dd I_*,$$
for any relevant $\beta$, since $\mathbf{m}_{i}^{\beta}(T)=\mathcal{I}_\beta/\mathcal{I}_0$. Thanks to Lemma~\ref{l:exprmmphi}, the change of variable $I_*\mapsto E_i'=I+I_*$ leads to
$$\mathcal{I}_\beta =\int_0^{+\infty} e^{-E_i/T} \, {E_i}^{\beta} \, \mathfrak{m}_{\varphi}'(E_i) \, \dd E_i = \frac{\Gamma(\delta/2)^2}{\Gamma(\delta)} \int_0^{+\infty} e^{-E_i/T} \, {E_i}^{\beta + \delta - 1} \, \dd E_i,$$
and subsequently to
$$\mathcal{I}_\beta =\frac{\Gamma(\delta/2)^2 \, \Gamma(\beta+\delta)}{\Gamma(\delta)} \, T^{\beta+\delta}.$$
\end{proof}


\subsection*{Acknowledgements} 
The authors are grateful to Professor Kazuo Aoki for useful discussions.


\bibliographystyle{abbrv}
\bibliography{biblio}

\end{document}